\documentclass[12pt]{amsart}
\usepackage{amsmath, amsthm, amssymb}
\usepackage{amssymb}
\usepackage{amsfonts}
\usepackage{amscd}
\usepackage{mathrsfs}
\usepackage{latexsym}
\usepackage{amsmath}
\usepackage{txfonts}
\usepackage{tikz}
\usepackage{graphicx}
\usepackage{subfig}
\usepackage{tikz-cd}
\usepackage{blindtext}

\usepackage[all, cmtip]{xy}

\newtheorem{theorem}{Theorem}[section]
\newtheorem{lemma}[theorem]{Lemma}
\newtheorem{corollary}[theorem]{Corollary}

\theoremstyle{definition}

\theoremstyle{remark}

\newcommand{\sshf}[1]{\mathscr{O}_{#1}}
\newcommand{\shf}[1]{\mathscr{#1}}

\newcommand{\iso}{\simeq}
\newcommand{\ses}[3]{0\rightarrow#1\rightarrow#2\rightarrow#3\rightarrow{0}}

\newcommand{\dlc}[1]{\omega^{\bullet}_{#1}}

\newcommand{\dhom}[2]{R\text{Hom}(#1, #2)}

\numberwithin{equation}{section}

\begin{document}
\allowdisplaybreaks
\title[Kodaira vanishing for singular varieties revisited] {Kodaira
  vanishing on singular varieties revisited}

\author{Donu Arapura}
\author{Lei Song}

\address{Department of Mathematics, Purdue University,
150 N. University Street, West Lafayette, IN 47907, U.S.A.}
\email{arapura@math.purdue.edu}

\address{School of Mathematics, Sun Yat-sen University, Guangzhou, Guangdong 510275, P.R. China}
\email{songlei3@mail.sysu.edu.cn}

\thanks{The first author was supported in part by a grant from the
  Simons Foundation. }

\date{}

\dedicatory{}

\keywords{}

\begin{abstract}
We correct the proof and slightly strengthen a Kodaira-type vanishing
theorem for singular varieties originally due to Jaffe and the first
author. Specifically, we show that if $L$ is a nef and big line bundle
on a projective variety of characteristic zero, the $i^{\text{th}}$ cohomology of
$L^{-1}$ vanishes for $i$ in a range determined by the depth and dimension of the singular locus.
\end{abstract}

\maketitle

\section{Introduction}

Some years ago, David Jaffe and the first author considered a version of Kodaira's
vanishing theorem for a singular variety, where the bound involves its
depth and dimension of the singular locus \cite[Proposition 1.1]{ArapuraJaffe89}; a  slight
strengthening of this says:

\begin{theorem}\label{AJvanishing}
  Let $X$ be a projective variety over an algebraically closed field
  of characteristic 0. Let $k\ge0$ be an integer. Suppose
  $k<\text{codim}(X_{sing})$ and $X$ satisfies Serre's condition $S_{k+1}$. Then for any nef and big line bundle $L$ on $X$, we have
\begin{equation*}
    H^k(X, L^{-1})=0.
\end{equation*}
\end{theorem}

Unfortunately, the argument given in \cite{ArapuraJaffe89} contains a
gap, because  it implicitly assumes that a general hyperplane section of an $S_{k+1}$ scheme remains
$S_{k+1}$. The purpose of this note is to fix the gap.
We use a different strategy involving Grothendieck duality. A small irony is  that this is
close to the  original approach that Jaffe and the first author used, although this proof never made it
to the final published version.

{\it Acknowledgments.} The second author would like to thank Osamu Fujino for pointing out \cite{EV84}.

\section{Notation and conventions}
We fix some terminology and notation.
\vspace{0.05cm}

(2.1) Given a projective scheme $X$ over a field, it admits a dualizing complex $\dlc{X}\in D^{+}(\text{QCoh})(X))$ with coherent cohomology sheaves \cite{Har66}. For a normalized $\dlc{X}$, the sheaf $\omega_X=h^{-\dim X}(\dlc{X})$ is called a dualizing sheaf.
\vspace{0.05cm}

(2.2) For a complex $\shf{F}^{\bullet}$ of sheaves on $X$,
$h^i(\shf{F}^{\bullet})$ denotes its $i$-th cohomology sheaf and
$\mathbb{H}^i(\shf{F}^{\bullet})$ denotes the hypercohomology $R^i\Gamma(\shf{F}^{\bullet})=h^i(R\Gamma(\shf{F}^{\bullet}))$. We will use $H^i(-)$ to denote the usual cohomology for sheaves. $R\text{Hom}(\shf{F}, -)$ is the derived functor of $\text{Hom}(\shf{F}, -)$.
\vspace{0.05cm}

(2.3) Let $X$ be a projective variety and given a line bundle $L$. We
say $L$ is nef, if for any proper morphism $f: C\rightarrow X$ from a
smooth projective curve over the base field, $\deg{f^*L}\ge 0$. We say
$L$ is big, if for any resolution of singularities $\mu: X'\rightarrow
X$, the pullback $\mu^*L$ is big on $X'$. Note that this is independent of $X'$.

(2.4) Let $X$ be a projective variety and $L$ a big line bundle. There exists a proper closed subset $V=V(L)\subset X$ with the property that if $Y\subset X$ is a subvariety not contained in $V$, then $L|_Y$ is also big (\cite[Corollary 2.2.11]{Laz04I}).

\section{Proof of theorem}
We work over an algebraically closed field of characteristic zero.

We have a second quadrant spectral sequence
\begin{equation}\label{spectral sequence}
    E^{i, j}_{2}=H^j(X, h^{i}(\dlc{X})\otimes L)\Rightarrow \mathbb{H}^{i+j}(X, \dlc{X}\otimes L).
\end{equation}
The key technical result is

\begin{theorem}\label{degeneracy}
Let $X$ be a projective variety and $L$ a nef and big line bundle on
$X$. Then for any $i, j$ such that $j>0$ and
$i+j>-\text{codim}(X_{\text{Sing}})$, we have
\begin{equation*}
    E^{i, j}_{\infty}=0.
\end{equation*}
\end{theorem}

Chasing the spectral sequence (\ref{spectral sequence}) more carefully, one can say something more.

\begin{corollary}\label{corollary}
Let $X$ be a projective variety and $L$ a nef and big line bundle on $X$.
\begin{enumerate}
  \item The differential
   $H^0(X, h^{-1}(\dlc{X})\otimes L)\rightarrow H^2(X,
   h^{-2}(\dlc{X})\otimes L)$ of \eqref{spectral sequence}
   is zero.
  \item Suppose $\dim X\le 3$. Then for any $i, j$ such that $j>0,
    i+j>-\text{codim}(X_{\text{Sing}})$, we have  that
\begin{equation*}
    h^j(X, h^{i}(\dlc{X})\otimes L)=0.
\end{equation*}

\end{enumerate}
\end{corollary}

We start with some lemmas.

\begin{lemma}\label{Kodaira vanishing}
Let $X$ be a projective variety and $L$ a nef and big line bundle on $X$. Then for $j>\dim{(X_{\text{Sing}})}$ and any $i\in \mathbb{Z}$,
\begin{equation*}
    H^j(X, h^{i}(\dlc{X})\otimes L)=0.
\end{equation*}
\end{lemma}
\begin{proof}
Let $n=\dim X$. For $i>-n$, the coherent sheaf $h^i(\dlc{X})$ is supported on $X_{\text{Sing}}$, so the vanishing is automatic. For $i=-n$, $h^{-n}(\dlc{X})\iso \omega_X$. Take a resolution of singularities $f: Y\rightarrow X$. The natural morphism $f_*\omega_Y\rightarrow \omega_X$ is injective, because $f_*\omega_Y$ is torsion free and the map is generically isomorphic. Consider the exact sequence
\begin{equation*}
    \ses{f_*\omega_Y}{\omega_X}{Q},
\end{equation*}
where the quotient sheaf $Q$ is supported on $X_{\text{Sing}}$. By \cite[Th{\'e}or{\`e}me 3.1]{EV84}, $H^j(X, f_*\omega_Y\otimes L)=0$ for all $j>0$ (The vanishing also follows from Grauert-Riemenschneider vanishing and Kawamata-Viehweg vanishing). So the assertion follows.
\end{proof}

\begin{lemma}\label{base case}
If $X$ is a 2-dimensional projective variety and $L$ a nef and big
line bundle on $X$. Then for any $i, j$ such that $j>0,
i+j>-\text{codim}(X_{\text{Sing}})$, we have
\begin{equation*}
    H^j(X, h^i(\dlc{X})\otimes L)=0.
\end{equation*}
\end{lemma}
\begin{proof}
If $\text{codim}(X_{\text{Sing}})=1$, by Lemma \ref{Kodaira vanishing},
\begin{equation*}
    E^{i, j}_2=0
\end{equation*}
if $j\ge 2$. By Grothendieck duality \cite{Har66},
$H^0(X,\dlc{X}\otimes L)= H^0(X, L^{-1})=0$. Therefore $E^{i, j}_{\infty}=0$ for $i+j=0$. This implies that $E^{-1, 1}_2=0$.

If $\text{codim}(X_{\text{Sing}})=2$, by Lemma \ref{Kodaira vanishing} again,
\begin{equation*}
    E^{i, j}_2=0
\end{equation*}
if $j>0$.
\end{proof}

\begin{proof}[Proof of Theorem \ref{degeneracy}]
We proceed by induction on $n=\dim X$. The $n=2$ case is proved in Lemma \ref{base case}. Fix a projective variety $X$ with $n\ge 3$ and a big and nef line bundle $L$ on $X$. Put $k=-(i+j)$, so $k<\text{codim}(X_{\text{Sing}})$. Suppose the statement is true for any projective varieties of dimension $n-1$. Fix a sufficiently ample Cartier divisor $H$ and $D\in |H|$ with the following properties:
 \begin{itemize}
   \item[(P1)] $D$ is integral and does not contain a component of
     $X_{sing}$;
   \item[(P2)] $L|_D$ is big and nef;
   \item[(P3)] $H^i(X, h^j(\dlc{X})\otimes \sshf{X}(H)\otimes L)=0$ for all $i>0$ and $j\in \mathbb{Z}$.
 \end{itemize}
(P1) is by the Bertini theorem, and (P2) is by (2.4). (P3) is true because $h^j(\dlc{X})$ are coherent and only finitely many of them are nonzero. 

Consider the short exact sequence of $\sshf{X}$-modules
\begin{equation*}
    \ses{\sshf{X}(-H)\otimes L^{-1}}{L^{-1}}{i_*\sshf{D}\otimes L^{-1}}.
\end{equation*}
By applying $\dhom{-}{\dlc{X}}=R\Gamma\circ R\shf{H}om(-, \dlc{X})$ and taking the cohomology, we get the exact sequence of vector spaces
\begin{eqnarray*}
    \cdots\rightarrow {h}^{-k}(\dhom{i_*\sshf{D}\otimes L^{-1}}{ \dlc{X}})\rightarrow
    {h}^{-k}(\dhom{L^{-1}}{\dlc{X}})\rightarrow \\
    {h}^{-k}(\dhom{\sshf{X}(-H)\otimes L^{-1}}{\dlc{X}})\rightarrow \cdots
\end{eqnarray*}
Using the fact that $L$ is invertible and the Grothendieck duality
for $i: D\hookrightarrow X$ \cite{Har66}, we arrive at
\begin{equation}\label{exact sequence of hypercohomology}
    \cdots\rightarrow \mathbb{H}^{-k}(\dlc{D}\otimes L|_D)\rightarrow \mathbb{H}^{-k}(\dlc{X}\otimes L)\rightarrow \mathbb{H}^{-k}(\dlc{X}\otimes\sshf{X}(H)\otimes L)\rightarrow \cdots
\end{equation}
The hypercohomology above can be computed by a spectral sequence (\ref{spectral sequence}):
\begin{equation*}
    'E^{p, q}_2:=H^q(D, h^p(\dlc{D})\otimes \otimes L|_D)\Rightarrow \mathbb{H}^{p+q}(\dlc{D}\otimes L|_D).
\end{equation*}
Similarly we denote by $E^{\bullet, \bullet}$ the spectral sequence for $\mathbb{H}^{\bullet}(\dlc{X}\otimes L)$ and $''E^{\bullet, \bullet}$ for $\mathbb{H}^{\bullet}(\dlc{X}\otimes\sshf{X}(H)\otimes L)$ respectively.

By (P3), $''E^{\bullet, \bullet}$ degenerates at the second page to
give an isomorphism
\begin{equation*}
    \mathbb{H}^{-k}(\dlc{X}\otimes\sshf{X}(H)\otimes L)\iso ''E^{-k, 0}_{\infty}\iso H^0(X, h^{-k}(\dlc{X})\otimes\sshf{X}(H)\otimes L).
\end{equation*}
(P1) implies that $\text{codim}(D_{\text{Sing}})\ge\text{codim}(X_{\text{Sing}})$. By the induction hypothesis, we have for $q>0$ and $p+q\ge -k$,
\begin{equation*}
    'E^{p, q}_{\infty}=0.
\end{equation*}
By looking at the $q$-th graded pieces with respect to the natural filtration of the hypercohomology in (\ref{exact sequence of hypercohomology}), we deduce that for $q>0$ and $p+q\ge -k$,
\begin{equation*}
    E^{p, q}_{\infty}=0,
\end{equation*}
as desired.
\end{proof}

Keep the notation introduced above in the rest of the section.

\begin{proof}[Proof of Theorem \ref{AJvanishing}]
By  Grothendieck duality  \cite{Har66} and Theorem \ref{degeneracy}, we have
\begin{equation*}
    H^k(X, L^{-1})^{\vee}\iso \mathbb{H}^{-k}(\dlc{X}\otimes L)\iso E^{-k, 0}_{\infty}\subseteq \cdots \subseteq E^{-k, 0}_2=H^0(X, h^{-k}(\dlc{X})\otimes L).
\end{equation*}
Thanks to Serre's condition $S_{k+1}$, $h^{-k}(\dlc{X})=0$; so the assertion follows.
\end{proof}

\begin{proof}[Proof of Corollary \ref{corollary}]
(1) Consider the commutative diagram
$$\xymatrix{
E^{-1, 0}_{\infty} \ar@{^(->}[d] \ar[r] &  H^0(X, h^{-1}(\dlc{X})\otimes \sshf{X}(H)\otimes L)\ar@{=}[d]\ar[r] & 0,\\
H^0(X, h^{-1}(\dlc{X})\otimes L)\ar[r]^{\alpha} &  H^0(X, h^{-1}(\dlc{X})\otimes \sshf{X}(H)\otimes L) & }$$
where the first row is exact, coming from (\ref{exact sequence of hypercohomology}); and the left column map is an inclusion because
\begin{equation*}
    E^{-1, 0}_{\infty}\subseteq\cdots \subseteq E^{-1, 0}_{2}=H^0(X, h^{-1}(\dlc{X})\otimes L).
\end{equation*}
We claim that the map $\alpha$ is injective for general $D$. In fact, if $D\in |H|$ is general such that $D$ does not contain any scheme defined by an associate prime of $h^{-1}(\dlc{X})$, then we have the exact sequence
\begin{equation*}
    \ses{h^{-1}(\dlc{X})\otimes L}{h^{-1}(\dlc{X})\otimes\sshf{X}(H)\otimes L}{h^{-1}(\dlc{X})\otimes\sshf{D}(H)\otimes L}.
\end{equation*}
Thus $\alpha$ is injective. It then follows that
\begin{equation*}
    E^{-1, 0}_{\infty}\iso H^0(X, h^{-1}(\dlc{X})\otimes L),
\end{equation*}
which in turn implies that
\begin{equation*}
    H^0(X, h^{-1}(\dlc{X})\otimes L)\rightarrow H^2(X, h^{-2}(\dlc{X})\otimes L).
\end{equation*}
is zero.

(2) The case $\dim X=2$ has been treated in Lemma \ref{base
  case}. When $\dim X=3$, we just show that
$H^2(X, h^{-2}(\dlc{X})\otimes L)=0$ when
$\text{codim}(X_{\text{Sing}})=1$. The other cases are left  to the interested reader.

By Theorem \ref{spectral sequence}, we deduce that the natural map
\begin{equation*}
    H^0(X, h^{-1}(\dlc{X})\otimes L)\rightarrow H^2(X, h^{-2}(\dlc{X})\otimes L)
\end{equation*}
surjects. But by (1) above, the map is indeed zero; so the assertion follows.
\end{proof}

\section{Comments}
(4.1) By Serre's normality criterion, Theorem \ref{AJvanishing} implies that
\begin{corollary}[Mumford {\cite{mumford}}]\label{cor:mumford}
  Suppose that $X$ is a normal projective variety of dimension at least 2, defined over a field of characteristic zero, and that $L$ is an  ample line bundle.
  Then $H^1(X, L^{-1})=0$.
\end{corollary}
This statement was in fact the original inspiration for Theorem \ref{AJvanishing}. In the same paper, Mumford goes on to give an easy
two dimensional counterexample to this statement in positive
characteristic. This was decade before Raynaud \cite{raynaud} gave a rather
different counterexample to the usual Kodaira vanishing for smooth
projective varieties.

(4.2) The depth condition is essential for the validity of Theorem \ref{AJvanishing}, while the condition on the dimension of the singular locus is not. Consider a projective variety $X$ with rational singularities. It is well-known that for any $i<\dim X$ and $L$ ample line bundle, the vanishing holds
\begin{equation*}
    H^i(X, L^{-1})=0.
\end{equation*}
To exhibit such a variety but with relatively small codimension of the singular locus, take a smooth projective variety $Y$ of dimension $n$ with the property that $H^i(Y, \sshf{Y})=0$ for all $i>0$, and fix a sufficiently ample line bundle $M$ on $Y$. Then the secant variety
\begin{equation*}
    \Sigma=\Sigma(Y, M)\subset\mathbb{P}(H^0(Y, M))
\end{equation*}
is $2n+1$ dimensional with rational singularities (cf.~\cite{ChouSong18}), and the singular locus is precisely $X\subset\Sigma$. Thus $\text{codim}(\Sigma_{\text{Sing}})=n+1<2n+1$.

\bibliography{bibliography}{}

\begin{thebibliography}{Mum67}

\bibitem[AJ89]{ArapuraJaffe89}
Donu Arapura and David~B Jaffe.
\newblock On {K}odaira vanishing for singular varieties.
\newblock {\em Proceedings of the American Mathematical Society},
  105(4):911--916, 1989.

\bibitem[CS17]{ChouSong18}
Chih-Chi Chou and Lei Song.
\newblock Singularities of secant varieties.
\newblock {\em International Mathematics Research Notices}, 2018(9):2844--2865,
  2017.

\bibitem[EV84]{EV84}
H~Esnault and E~Viehweg.
\newblock Rev{\^e}tements cycliques. {II} (autour du th{\'e}or{\`e}me
  d¡¯annulation de {J}. {K}oll{\'a}r).
\newblock {\em G{\'e}om{\'e}trie alg{\'e}brique et applications, II (La
  R{\'a}bida, 1984)}, 23:81--96, 1984.

\bibitem[Har66]{Har66}
Robin Hartshorne.
\newblock {\em Residues and duality}.
\newblock Springer-Verlag, Berlin, 1966.

\bibitem[Laz04]{Laz04I}
Robert Lazarsfeld.
\newblock {\em Positivity in algebraic geometry {I}: Classical setting: line
  bundles and linear series}, volume~48.
\newblock Springer, 2004.

\bibitem[Mum67]{mumford}
David~B Mumford.
\newblock Pathologies {III}.
\newblock {\em American Journal of Mathematics}, 89(1):94--104, 1967.

\bibitem[Ray78]{raynaud}
Michel Raynaud.
\newblock Contre-exemple au ``vanishing theorem" en caract\'eristique $p>0$.
  {C}.{P}. {R}amanujam-a tribute.
\newblock {\em Tata Institute of Fundamental Research, Studies in Mathematics},
  8:273--278, 1978.

\end{thebibliography}
\bibliographystyle{alpha}

\end{document}